\documentclass[reqno,12pt,letterpaper]{amsart}
\usepackage{amsmath,amssymb,amsthm,graphicx,mathrsfs,url}
\usepackage{subfigure} 
\usepackage[usenames,dvipsnames]{color}
\usepackage[colorlinks=true,linkcolor=Red,citecolor=Green]{hyperref}

\def\?[#1]{\textbf{[#1]}\marginpar{\Large{\textbf{??}}}}

\newtheorem{theo}{Theorem}
\newtheorem{prop}{Proposition}[section]

\newtheorem{lem}[prop]{Lemma}

\theoremstyle{remark}
\newtheorem{rem}{Remark}
\numberwithin{equation}{section}

\DeclareMathOperator{\supp}{supp}

\newcommand{\DOM}{\mathcal{D}}
\newcommand{\BINT}{{B(0,R_0)}}
\newcommand{\Bext}{{\mathbb{R}^n\backslash B(0,R_0)}}

\renewcommand{\Im}{\operatorname{Im}}
\renewcommand{\Re}{\operatorname{Re}}

\title[From quasimodes to resonances]
{From quasimodes to resonances: exponentially decaying perturbations.}
\author{Oran Gannot}
\email{ogannot@math.berkeley.edu}
\address{Department of Mathematics, Evans Hall, University of California,
Berkeley, CA 94720, USA}

\begin{document}

\begin{abstract}
We consider self-adjoint operators of black-box type which are exponentially close to the free Laplacian near infinity, and prove an exponential bound for the resolvent in a strip away from resonances. Here the resonances are defined as poles of the meromorphic continuation of the resolvent between appropriate exponentially weighted spaces. We then use a local version of the maximum principle to prove that any cluster of real quasimodes generates at least as many resonances, with multiplicity, rapidly converging to the quasimodes. 
\end{abstract}

\maketitle


\section{Introduction}

It is expected that for open systems, trapping of classical trajectories produces scattering resonances close to the real axis; this is often referred to as the Lax-Phillips  conjecture, see \cite[Section V.3]{Lax:1989}. When trapping is weak, for instance in the sense of hyperbolicity, the general conjecture is not true as shown by Ikawa \cite{Ikawa:1982}. For an account of recent results about resonances near the real axis under weak trapping, see the review by Wunsch \cite{Wunsch:2012}. On the other hand, when the trapping is sufficiently strong so that a construction of real quasimodes is possible, the works of Stefanov--Vodev \cite{Stefanov:1996}, Tang--Zworski \cite{Tang:1998}, and Stefanov \cite{Stefanov:1999} show that there exist resonances close to the quasimodes. These results were established in the setting of compactly supported perturbations, or more generally for perturbations which are dilation analytic near infinity \cite{Sjostrand:1991}, \cite{Sjostrand:1997}.
   
Complementary to the aforementioned results, in this note we prove analogues for `black box' operators which are exponentially close to the free Laplacian at infinity. More precisely, we allow both metric and potential perturbations of the Laplacian outside a compact set (the black box), but require only minimal assumptions on the operator in the black box. Standard techniques give a meromorphic continuation of the exponentially weighted resolvent through the real axis to a strip whose width is of size $O(h)$; the choice of exponential weight and the width of the strip depend on the decay rate of the perturbation. We then apply a complex analytic framework, summarized for example in \cite{Petkov:2001}, to deduce an exponential a priori bound on the weighted resolvent away from resonances. 

A typical application of such an exponential bound (well-established in \cite{Stefanov:1996}, \cite{Tang:1998}, \cite{Stefanov:1999}, \cite{Stefanov:2005}) is to show that any family of sufficiently independent quasimodes generates at least as many resonances, counting multiplicity; these resonances converge rapidly not only to the real axis, but to the quasimodes --- see Theorem \ref{thm:existenceofres} for a precise statement. The general assumptions are presented in Section \ref{sect:blackboxmodel}.   

One motivation for this work comes from a recent investigation of resonances for Schwarzschild--AdS black holes, where quasimodes have been constructed \cite{Gannot:2014}, \cite{Holzegel:2013}. Due to the spherical symmetry in that setting, the stationary wave operator $P$ decomposes as a sum of one-dimensional operators $P_\ell$ on a half-line, which are just restrictions to spaces of spherical harmonics with angular momentum $\ell$. Each $P_\ell$ is a self-adjoint perturbation of the Laplacian by an exponentially decaying potential which is singular near the origin --- the results of this paper imply that the resolvent $R_\ell(\sigma)$ of $P_\ell$ has a meromorphic continuation through the real axis. Although meromorphy of each one-dimensional resolvent does not imply meromorphy for the global resolvent (this requires uniform control as $\ell \rightarrow \infty$ and has recently been established in the Schwarzschild--AdS setting by Warnick \cite{Warnick:2013}), the results of this paper do imply the existence of a sequence of poles $\sigma_\ell$ for $R_\ell(\sigma)$ satisfying
\[
0 < -\Im \sigma_\ell < C e^{-\ell/C} \text{ for $\ell$ sufficiently large}.
\] 
We also remark that in the Schwarzschild--AdS case the effective potential is dilation analytic, so the results of \cite{Sjostrand:1997} indeed apply. One advantage to the approach taken here is that the exponential decay of the potential remains stable under small (radial, static) perturbations of the Schwarzschild--AdS metric.

\subsection{Free resolvent}

We begin by gathering several results about the free resolvent. The Laplacian $-\Delta$ on $\mathbb{R}^n$ with domain $H^2(\mathbb{R}^n)$ is self-adjoint and we denote by $R_0(\sigma)$ the free resolvent
\[
R_0(\sigma) = (-\Delta - \sigma^2)^{-1} : L^2(\mathbb{R}^n) \rightarrow H^2(\mathbb{R}^n), \, \Im\sigma >0.
\]

Choose $\varphi\in C^\infty(\mathbb{R}^n)$ with the property that $\varphi(x) = |x|$ for $|x|$ large enough. If $\mathcal{A}$ denotes some function space, we will use the notation $\mathcal{A}_\gamma = e^{-\gamma\varphi}\mathcal{A}$ for its weighted counterpart. We will also freely move between the equivalent notions
\[
T: \mathcal{A}_\alpha \rightarrow \mathcal{B}_\beta \Longleftrightarrow e^{\beta\varphi}Te^{-\alpha\varphi}: \mathcal{A} \rightarrow \mathcal{B},
\]
depending on convenience. 

Our starting point is the well known fact \cite{McLeod:1967} that for each $\gamma>0$ the weighted resolvent 
\[
e^{-\gamma \varphi} R_0(\sigma) e^{-\gamma \varphi} : L^2(\mathbb{R}^n) \rightarrow L^2(\mathbb{R}^n)
\]
extends holomorphically across $\Re \sigma > 0$ as a bounded operator to the strip $\Im \sigma > -\gamma$, with the usual caveats in even dimensions when winding around the origin. We also have the standard representation,
\begin{equation} \label{eq:spectralmeasure}
e^{-\gamma \varphi} R_0(\sigma) e^{-\gamma \varphi} = e^{-\gamma \varphi} R_0(-\sigma) e^{-\gamma \varphi} +  \sigma^{n-2} e^{-\gamma \varphi} M(\sigma)e^{-\gamma \varphi}
\end{equation}
whenever $\Re \sigma >0 $ and $-\gamma < \Im\sigma < 0$. Here $M(\sigma)$ is the operator with kernel
\[
M(\sigma,x,y) = (i/2)(2\pi)^{-n+1} \int_{\mathbb{S}^{n-1}} e^{i\sigma\left<\omega,x-y\right>}d\omega.
\]
We can also write
\begin{equation} \label{eq:MandPhi}
M(\sigma) = (i/2)(2\pi)^{-n+1} \Phi^t(\sigma) \Phi(-\sigma),
\end{equation}
where $\Phi(\sigma) : L^2(\mathbb{R}^n) \rightarrow L^2(\mathbb{S}^{n-1})$ has kernel $\Phi(\sigma,\omega,x) = e^{i\sigma\left<\omega,x\right>}$ and $\Phi^t: L^2(\mathbb{S}^{n-1}) \rightarrow L^2(\mathbb{R}^n)$ has the transposed kernel.

The following two lemmas establish standard polynomial bounds for the free resolvent in the case of exponential weights.

\begin{lem} \label{lem:Mdecay} For each $\epsilon > 0$ there exists $C = C(\epsilon)>0$ such that whenever $| \Im \sigma | < \gamma - \epsilon$ and $\Re \sigma \geq 1$,
\[
\| e^{-\gamma \varphi} M(\sigma) e^{-\gamma \varphi} \|_{L^2(\mathbb{R}^n) \rightarrow L^2(\mathbb{R}^n)} < C |\sigma|^{1-n}.
\]
\end{lem}
\begin{proof}
The proof is adapted from \cite{Burq:2002}. First note that the Fourier transform $\mathcal{F}(e^{-\gamma\varphi})(\xi)$ extends holomorphically to the strip $\{ \xi \in \mathbb{C}^n: |\Im \xi| < \gamma -\epsilon \}$ and 
\begin{equation} \label{eq:fourier}
|\mathcal{F}(e^{-\gamma\varphi})(\xi)| < C_N \left < \xi \right>^{-N}
\end{equation}
in the strip for each $N$. In light of \eqref{eq:spectralmeasure}, \eqref{eq:MandPhi} it suffices to prove that 
\[
 \| \Phi(\sigma) e^{-\gamma\varphi} \|_{L^2(\mathbb{R}^n) \rightarrow L^2(\mathbb{S}^{n-1})}  < C |\sigma|^{(1-n)/2},
\]
which by Plancherel's theorem is equivalent to the same estimate for the composition $\left(\Phi(\sigma) e^{-\gamma\varphi}\right)\circ \mathcal{F}$. The operator $\left(\Phi(\sigma) e^{-\gamma\varphi}\right)\circ \mathcal{F}$ has kernel $\mathcal{F}(e^{-\gamma\varphi})(\sigma \omega -\xi)$. By Schur's lemma it suffices to obtain an estimate of the form
\[
\sup_{\xi\in \mathbb{R}^n} \int_{\mathbb{S}^{n-1}} |\mathcal{F}(e^{-\gamma\varphi})(\sigma \omega -\xi)| d\omega < C |\sigma|^{1-n},
\]
since in the other direction we may use \eqref{eq:fourier} to obtain the trivial estimate
\[
\sup_{\omega\in \mathbb{S}^{n-1}} \int_{\mathbb{R}^{n}} |\mathcal{F}(e^{-\gamma\varphi})(\sigma \omega -\xi)| d\xi < C.
\]
Write $\xi$ as $\xi = \left<\xi,\omega\right> \omega + \xi^\perp(\omega)$ where $\left< \xi^\perp(\omega),\omega\right>= 0$. Then by \eqref{eq:fourier} we are left estimating
\[
\int_{\mathbb{S}^{n-1}} (1+ |\left<\xi, \omega \right> - \Re \sigma| + |\xi^{\perp}(\omega)|)^{-N} d\omega.
\]
Fix $\xi \in \mathbb{R}^n$ and $\delta>0$, and decompose the sphere into two sets, 
\[
U = \{ \omega\in \mathbb{S}^{n-1}: |\left<\xi, \omega \right> - \Re \sigma| < \delta \Re \sigma,\; |\xi^{\perp}(\omega)| < \delta \Re \sigma \}
\]
and its complement $U^c$. The integral over $U^c$ is of the order $O(|\Re \sigma|^{-\infty})$ so it suffices to examine the integral over $U$. 

Observe that unless $\Re \sigma$ is comparable to $|\xi|$, the set $U$ is empty. Indeed, if $\omega\in U$ then $(1-\delta)\Re\sigma < \left<\xi, \omega \right>  < (1+\delta)\Re\sigma$. Hence 
\[
\frac{\Re \sigma}{2} < (1-\delta)\Re\sigma < \left<\xi, \omega \right>  \leq |\xi|,
\]
while on the other hand 
\[
|\xi|^2 = |\left<\xi, \omega \right> |^2 + |\xi^\perp(\omega)|^2 < 3(\Re\sigma)^2
\] 
for $\delta$ sufficiently small.

Write a typical point of $\mathbb{R}^n$ as $(y,y')$ where $y \in \mathbb{R}^{n-1}$ and $y' \in \mathbb{R}$. By a rotation we may assume that $\xi = (0,|\xi|)$. In that case $U$ is contained in the upper hemisphere, in a cap around $|\xi|^{-1}\xi = (0,1)$ whose size is independent of $\xi$. This is true since $\omega \in U$ implies
\[
\left< |\xi|^{-1} \xi,\omega \right> > \frac{1}{2\sqrt{3}} > 0.
\]

We then parametrize the upper hemisphere $\mathbb{S}^{n-1}_+$ (which contains $\xi$) by the diffeomorphism
\[
p: \mathbb{R}^{n-1} \rightarrow \mathbb{S}^{n-1}_+, \quad y \mapsto \frac{(y,|\xi|)}{|(y,|\xi|)|}.
\]
Whenever $y \in p^{-1}(U)$ we have 
\[
|\xi^{\perp}(p(y))| \geq  |y|.
\] 
To see this, compute
\[
|\xi^{\perp}(p(y))|^2 = |\xi|^2 - |\left<\xi,p(y)\right>|^2 = |\xi|^2 - \frac{|\xi|^4}{|y|^2+|\xi|^2} = |y|^2 \frac{|\xi|^2}{|y|^2+|\xi|^2} \geq |y|^2. 
\]
Furthermore, the Jacobian satisfies
\[
| \partial p / \partial y | = O(|\xi|^{1-n}).
\]
We can now bound the integral over $U$ by
\begin{align*}
\int_U (1 + |\xi^{\perp}(\omega)|)^{-N} d\omega &= \int_{p^{-1}(U)} (1+ |\xi^\perp(p(y))|)^{-N}\, |\partial p /\partial y | dy \\
&\leq C_1 |\Re\sigma|^{1-n} \int_{\mathbb{R}^{n-1}} (1+|y|)^{-N} dy \leq C_2 |\Re\sigma|^{1-n}
\end{align*}
for $N$ large enough. In the final step we used that $|\xi|$ and $\Re\sigma$ were comparable.

\end{proof}

\begin{lem} \label{lem:freeresolventbounds}
For each $\epsilon > 0$ and $|\alpha|\leq 2$ there exists $C_\alpha = C_\alpha(\gamma,\epsilon)$ such that whenever $\Im \sigma > -\gamma + \epsilon$ and $\Re \sigma \geq 1$,      
\[
\| D^\alpha \left(e^{-\gamma\varphi}R_0(\sigma)e^{-\gamma\varphi}\right)\|_{L^2 \rightarrow L^2} \leq C_\alpha |\sigma|^{|\alpha|-1}.
\]
\end{lem}
\begin{proof}
$(1)$ First we handle the case $|\alpha|=0$ and $n>1$; see \cite{Rauch:1978} and \cite{Vodev:1994} for similar arguments. Let $U(t) = \cos (t \sqrt{-\Delta})$ denote the propagator for the Cauchy problem
\[
\begin{cases}
(\partial_t^2 - \Delta)U(t)f(x) = 0, \; (t,x) \in \mathbb{R} \times \mathbb{R}^n, \\
U(0)f(x) = f(x), \quad \partial_t U(0)f(x) = 0.
\end{cases}     
\]
For $\Im \sigma > 0$, write the resolvent
\[ \label{eq:resolventintermsofpropagator}
e^{-\gamma \varphi}R_0(\sigma)e^{-\gamma \varphi} = \frac{i}{\sigma} \int_0^\infty e^{i\sigma t} e^{-\gamma\varphi} U(t) e^{-\gamma\varphi} dt.
\]
Let $r_0$ be such that $\varphi(x) = |x|$ for $|x| \geq r_0$. Notice that $\|U(t)\|_{L^2\rightarrow L^2} \leq 1$ and 
\[
\|1_{\{|x| \geq t/4\}} e^{-\gamma\varphi} \|_{L^2(\mathbb{R}^n)\rightarrow L^2(\mathbb{R}^n)} \leq e^{-\gamma t/4}, \quad t\geq 4r_0.
\]
Writing
\begin{align*}
U(t) &= 1_{\{|x| < t/4\}} U(t) 1_{\{|x| < t/4\}} + 1_{\{|x| \geq t/4\}}  U(t) 1_{\{|x| < t/4\}} \\
&+ 1_{\{|x| < t/4\}} U(t)  1_{\{|x| \geq t/4\}} + 1_{\{|x| \geq t/4\}}  U(t)  1_{\{|x|\geq t/4\}},  
\end{align*}
we see the norms of the latter three terms are of size $O(e^{-\gamma t/4})$ after multiplication by $e^{-\gamma\varphi}$ on the left and right.
Hence we only need to estimate the norm of the operator with kernel
\[
1_{\{|x| < t/4\}}(x) e^{-\gamma\varphi(x)} U(t,x,y) e^{-\gamma\varphi(y)} 1_{\{|x| < t/4\}}(y),
\]
using explicit knowledge of the kernel $U(t,x,y)$. 

In odd dimensions, the kernel vanishes identically by the strong Huygens principle. In even dimensions, the kernel vanishes unless $|x|, |y| < t/4$ which implies $|x-y|<t/2$ and thus 
\[
|1_{\{|x| < t/4\}}(x) U(t,x,y) 1_{\{|x| < t/4\}}(y)| \leq Ct^{-n}, 
\]
again from explicit formulae for $U(t,x,y)$. Schur's lemma then gives
\[
\| 1_{\{|x| < t/4\}} e^{-\gamma\varphi} U(t) e^{-\gamma\varphi} 1_{\{|x| < t/4\}}\|_{L^2(\mathbb{R}^n)\rightarrow L^2(\mathbb{R}^n)} \leq C t^{-n}.
\] 
Therefore we see that the integral in \eqref{eq:resolventintermsofpropagator} actually converges for $\Im \sigma \geq 0$ with the uniform estimate
\[
\| e^{-\gamma \varphi}R_0(\sigma)e^{-\gamma \varphi}\|_{L^2\rightarrow L^2} \leq C |\sigma|^{-1}, \; \Im\sigma \geq 0 \text{ and } \Re \sigma \geq  1.
\]
The result for $-\gamma + \epsilon < \Im \sigma < 0$ follows immediately by reflection from \eqref{eq:spectralmeasure} and Lemma \ref{lem:Mdecay}.

$(2)$ In the case $\alpha = 0$ and $n=1$ one can simply apply Schur's lemma to the Schwartz kernel
\[
e^{-\gamma \varphi} R_0(x,y,\sigma) e^{-\gamma \varphi} = e^{-\gamma\varphi(x)}\frac{i e^{i\sigma|x-y|}}{\sigma} e^{-\gamma\varphi(y)}.
\]

The $|\alpha|=1,2$ cases follows from the $|\alpha|=0$ case by interpolation as in \cite[Lemma 3]{Zworski:1989}; we supply a proof for the reader's convenience. Consider first the case $|\alpha|=2$. By analytic continuation, if $u\in L^2(\mathbb{R}^n)$ then
\begin{equation} \label{eq:continuedresolventidentity}
\Delta R_0(\sigma)e^{-\gamma\varphi}u = -e^{-\gamma\varphi}u - \sigma^2 R_0(\sigma)e^{-\gamma\varphi}u
\end{equation}
and hence $R_0(\sigma) : L^2_\gamma \rightarrow H^2_{-\gamma}$ is bounded for $\Im \sigma > -\gamma$. Now choose $u\in L^2(\mathbb{R}^n)$ and set $f=R_0(\sigma)e^{-\gamma \varphi} u$. Then
\begin{equation} \label{eq:laplacianofweightedresolvent}
\Delta (e^{-\gamma \varphi} R_0(\sigma) e^{-\gamma \varphi} u) = (\gamma^2 |\nabla \varphi|^2 -\gamma \Delta \varphi)e^{-\gamma \varphi} f - 2\gamma \nabla \varphi \cdot (e^{-\gamma \varphi} \nabla f) + e^{-\gamma \varphi} \Delta f
\end{equation}
In light of \eqref{eq:continuedresolventidentity} it suffices to estimate the $L^2$ norm of $-\gamma \nabla \varphi \cdot (e^{-\gamma \varphi} \nabla f)$.
But since $\varphi$ has uniformly bounded derivatives, 
\[ 
\| \nabla \varphi \cdot (e^{-\gamma \varphi} \nabla f) \|^2_{L^2} \leq C \| e^{-\gamma \varphi} \nabla f \|^2_{L^2}. 
\]
We now integrate by parts and estimate
\begin{equation} \label{eq:nablaofweightedresolvent}
\|e^{-\gamma \varphi} \nabla f\|^2_{L^2} \leq 2 \int | \gamma \nabla \varphi|\,|e^{-\gamma \varphi} \nabla f|\,|e^{-\gamma \varphi} f| \,dx + \int |e^{-\gamma \varphi} \Delta f|\,|e^{-\gamma \varphi} f|  dx. 
\end{equation}
Applying the inequality $2ab \leq 2a^2 + \frac{1}{2}b^2$ to the integrand, the first term on the right hand side is bounded by
\[
\int 2 |\gamma \nabla \varphi|^2 \, |e^{-\gamma \varphi} f|^2 \, dx + \int \frac{1}{2} \, |e^{-\gamma \varphi} \nabla f|^2 \, dx
\]
while for the second term we use \eqref{eq:continuedresolventidentity}. We conclude that 
\[
\|e^{-\gamma \varphi} \nabla f\|^2_{L^2} \leq C(1+|\sigma|^2) \| e^{-\gamma \varphi} f \|^2_{L^2} + \|e^{-2\gamma\varphi} u\|^2_{L^2}.
\]
Returning to \eqref{eq:laplacianofweightedresolvent}, it follows that
\[
\| \Delta (e^{-\gamma \varphi} R_0(\sigma) e^{-\gamma \varphi} u) \|_{L^2} \leq C(1+|\sigma|^2)\|u\|_{L^2}.
\]
Moreover, \eqref{eq:nablaofweightedresolvent} actually shows
\[
\| \nabla (e^{-\gamma \varphi} R_0(\sigma) e^{-\gamma \varphi} u) \|_{L^2} \leq C\|u\|_{L^2}.
\]
\end{proof}  

We now introduce the semiclassical rescaling by setting $\lambda = h \sigma$. Let $R_0(\sigma,h)$ denote $(-h^2\Delta-\lambda^2)^{-1}$ and its corresponding analytic continuation. We are interested in $\lambda$ lying in a set of the form
\[
(a,b) + i((-\gamma+\epsilon)h,1)
\]
where $0 < a < b$. For the remainder of the paper, equip $H^k(\mathbb{R}^n)$ with the $h$-dependent norm $\|u\|^2_{H^k} = \sum_{|\alpha|\leq k}\| (hD)^\alpha u \|^2_{L^2}$.  Since $R_0(\lambda,h) = h^{-2}R_0(\textstyle{\frac{\lambda}{h}})$, it follows that we have uniform estimates
\[
\| R_0(\lambda,h)\|_{L^2_\gamma \rightarrow H^s_{-\gamma}} = O(h^{-1}),\; s=0,1,2,
\]
for $\lambda \in (a,b) + i((-\gamma+\epsilon)h,1)$.

\subsection{Black box model} \label{sect:blackboxmodel}

As our scattering problem, we consider exponentially decaying perturbations of the Laplacian outside a compact set, formulated in the black box setting as follows. Suppose $\mathcal{H}$ is a Hilbert space with an orthogonal decomposition 
\[
\mathcal{H} = \mathcal{H}_{R_0} \oplus L^2(\mathbb{R}^n\backslash B(0,R_0))
\] 
where $B(x,R) = \left\{y\in \mathbb{R}^n: \, |x-y| < R\right\}$ and $R_0$ is fixed. The orthogonal projections onto $\mathcal{H}_{R_0}$ and $L^2(\mathbb{R}^n\backslash B(0,R_0))$ will be denoted $1_{B(0,R_0)}u = u|_{B(0,R_0)}$ and $1_{\mathbb{R}^n\backslash B(0,R)}u = u|_{\mathbb{R}^n\backslash B(0,R_0)}$ for $u\in \mathcal{H}$. Note that any bounded continuous function $\chi\in C_b(\mathbb{R}^n)$ which is constant near $B(0,R_0)$ acts naturally on $\mathcal{H}$ by
\[
\chi u = C_0 u + (\chi-C_0)1_{\mathbb{R}^n\backslash B(0,R_0)} u
\]
where $\chi \equiv C_0$ near $B(0,R_0)$.   

Now consider an unbounded self-adjoint operator $P(h)$ on $\mathcal{H}$ with domain $\DOM \subset \mathcal{H}$ (independent of $h$ for simplicity) with the following properties. The domain is required to satisfy
\begin{itemize} \itemsep4pt
\item If $u \in \DOM$, then $1_{\mathbb{R}^n\backslash B(0,R_0)}u \in H^2(\mathbb{R}^n\backslash B(0,R_0))$;
\item If $u\in H^2(\mathbb{R}^n\backslash B(0,R_0))$ vanishes near $B(0,R_0)$, then $u\in \DOM$. 
\end{itemize}

We assume there exists a real-valued and uniformly positive-definite matrix $(a_{ij})$, along with a real-valued function $V$, (which are allowed to be $h$-dependent) such that
\begin{equation} \label{eq:longrangeproperty}
(P(h)u)|_\Bext = \left(-\sum_{i,j} (h\partial_{i}) a_{ij} (h\partial_{j}) + V \right)(u|_\Bext), \, u\in \DOM.
\end{equation}
Furthermore, we require that
\[
a_{ij}(x;h) \in C_b^\infty(\Bext); \quad V(x;h) \in C_b^\infty(\Bext),
\] 
with all derivatives uniformly bounded in $h$.

The perturbation is assumed to decay exponentially to the Laplacian in the sense that there exists $\gamma>0,\, \delta>0$ so that
\begin{equation} \label{eq:potentialdecay}
|a_{ij}(x;h)-\delta_{ij}| \leq C e^{-(2\gamma+\delta) |x|}; \quad |V(x;h)| \leq C e^{-(2\gamma+\delta) |x|},\, x\in \Bext.
\end{equation}
Finally, assume that the mapping 
\begin{equation} \label{eq:blackboxcompactness}
1_{B(0,R_0)} (P(h)+i)^{-1}: \mathcal{H} \rightarrow \mathcal{H}_{R_0}
\end{equation}
is compact.


Under these hypotheses, we show that 
\[
R(\lambda,h) = (P(h)-\lambda^2)^{-1},\quad  \Re \lambda >0, \, \Im \lambda > 0
\]
admits a meromorphic continuation to the strip $\Im \lambda > (-\gamma+\epsilon)h$ as an operator $\mathcal{H}_\gamma\rightarrow \mathcal{H}_{-\gamma}$. In order that the associated weighted space $\mathcal{H}_\gamma$ makes sense, we choose $\varphi\in C^\infty(\mathbb{R}^n)$ as above satisfying $\varphi \equiv 0$ near $\BINT$.

\begin{rem} All of the results in this note also apply to black box operators on the half-line $(0,\infty)$. For the most part this amounts to replacing the Laplacian on $\mathbb{R}^n$ with the Dirichlet Laplacian on $(0,\infty)$, and replacing $H^s(\mathbb{R}^n)$ with $H^s(0,\infty)\cap H^1_0(0,\infty)$. Estimates for the free resolvent on $(0,\infty)$ follow from those on $\mathbb{R}$ by odd reflection; all other necessary modifications should be clear.
\end{rem}

\subsection{Meromorphic continuation}

As a preliminary, arbitrarily extend $a_{ij}$ and $V$ to functions defined on all of $\mathbb{R}^n$ with the same properties as their original counterparts. Since the choice of extension has no bearing on the final result, we denote them by the same letters. Now define
\begin{align*}
\widetilde{P}(h) &= -\sum_{i,j} (h\partial_{i}) a_{ij} (h\partial_{j}) + V, \\
\widetilde{R}(\lambda,h) &= (\widetilde{P}(h)-\lambda)^{-1},\; \lambda^2 \notin \sigma(\widetilde{P}(h)).
\end{align*}
Since $\widetilde{P}(h)$ is uniformly elliptic, it is self-adjoint with domain $H^2(\mathbb{R}^n)$. We will also write $A(h)$ for the difference
\[
A(h) = \widetilde{P}(h) - (-h^2\Delta).
\]
The important fact about $A(h)$ is that it is bounded as a map $H^s_\alpha \rightarrow H^{s-2}_{\alpha+2\gamma+\delta}$ for each $s,\alpha\in \mathbb{R}$.
 
We will need information about the $L^2_\gamma\rightarrow H^s_\gamma$ mapping properties of $\widetilde{R}(\lambda,h)$ for $\lambda^2 \notin \sigma(\widetilde{P}(h))$. 
\begin{lem}
Fix an interval $(a,b)\Subset \mathbb{R}_+$. For each $\gamma>0$ there exists $T_0>0$ such that 
\[
\|e^{\gamma\varphi}\widetilde{R}(\lambda,h)e^{-\gamma\varphi}\|_{L^2\rightarrow H^s}  = O(|\Im \lambda|^{-1}), \; s=0,1,2
\]
uniformly for $\lambda\in (a,b)+i(T_0 h , 1)$.
\end{lem}
\begin{proof}
Conjugating $\widetilde{P}(h)$ by $e^{\gamma\varphi}$ yields
\[
e^{\gamma\varphi} \widetilde{P}(h) e^{-\gamma\varphi} = \widetilde{P}(h) + h^2 B, 
\]
where 
\[
B = \sum_{i,j} (2\gamma a_{ij}\partial_i\varphi) \partial_j - \gamma^2 a_{ij} \partial_i \varphi\, \partial_j \varphi + \gamma \partial_i(a_{ij}\partial_j\varphi)
\]
is a first order operator with uniformly bounded coefficients. It follows that for $\lambda^2 \notin \sigma(\widetilde{P}(h))$ (in particular for $\Im \lambda > 0$ and $\Re \lambda > 0$) we can write
\[
e^{\gamma\varphi} \widetilde{P}(h) e^{-\gamma\varphi} - \lambda^2 = (I + h^2 B \widetilde{R}(\lambda,h))(\widetilde{P}(h)-\lambda^2)
\]
Now $\widetilde{P}(h)$ is self-adjoint and elliptic, hence 
\[
\| u \|_{H^2} < C\|(\widetilde{P}(h)+i)u\|_{L^2}.
\]
It follows that for $\lambda \in (a,b) + i(0,1)$, 
\begin{equation} \label{eq:tilderesolvent}
\| \widetilde{R}(\lambda,h) \|_{L^2\rightarrow H^s} = O(|\Im \lambda|^{-1}),\; s=0,1,2.
\end{equation}
We immediately deduce that
\[
\|h^2 B \widetilde{R}(\lambda,h)\|_{L^2\rightarrow L^2} = O(h |\Im \lambda|^{-1}) \leq 1/2, \; \lambda\in (a,b) + i[T_0 h,1) 
\]
for $T_0>0$ large enough. In particular, $I+h^2 B \widetilde{R}(\lambda,h):L^2(\mathbb{R}^n)\rightarrow L^2(\mathbb{R}^n)$ is invertible for $\lambda\in (a,b) + i[T_0h,1)$ and
\[
e^{\gamma\varphi}\widetilde{R}(\lambda,h)e^{-\gamma\varphi} = \widetilde{R}(\lambda,h)(I+h^2 B \widetilde{R}(\lambda,h))^{-1}.
\]
This also shows that
\[
\|e^{\gamma\varphi}\widetilde{R}(\lambda,h)e^{-\gamma\varphi}\|_{L^2\rightarrow H^s} = O(|\Im \lambda|^{-1}), \; s=0,1,2 
\]
for $\lambda\in (a,b) + i[T_0 h,1)$.
\end{proof} 

The following lemma is useful in the proof of the meromorphic continuation. Equip $\DOM$ with the $h$-dependent norm
\[
\| u \|_\DOM = \| (P(h) + i)u \|_{\mathcal{H}}.
\]
Then it is easy to see that under the uniform boundedness conditions on the derivatives of $a_{ij}$ and $V$, the analog of \cite[Proposition 4.1]{Sjostrand:1991} remains true:
\begin{lem} \label{lem:uniformboundedness}
Suppose $\chi \in C^\infty_b(\mathbb{R}^n)$ has support disjoint from $\overline{\BINT}$. Then multiplication by $\chi$ is bounded $\DOM \rightarrow H^{2}(\mathbb{R}^n)$ and $H^{2}(\mathbb{R}^n)\rightarrow\DOM$ with a norm bounded independently of $h$.
\end{lem}
\begin{proof}
Consider first the map $\chi : \DOM \rightarrow H^2(\mathbb{R}^n)$. Since $\widetilde{P}(h)$ is elliptic, we have the a priori estimate
\begin{align*}
\| \chi u \|^2_{H^2(\mathbb{R}^n)} &\leq C_1 \left( \| \chi_1 \widetilde{P}(h) 1_{\Bext} u \|^2_{L^2(\Bext)} + \| \chi_1 1_{\Bext} u \|^2_{L^2(\Bext)} \right) \\ &\leq C_2 \| (P(h)+i)u \|^2_{\mathcal{H}},  
\end{align*}
where $\chi_1 \equiv 1$ on $\supp \chi$ and $\chi_1$ also has support disjoint from $\overline{\BINT}$. All the constants are independent of $h$. For the case $\chi : H^2(\mathbb{R}^n) \rightarrow \DOM$ this is equivalent to the uniform boundedness of $\widetilde{P}(h)$ on $H^2(\mathbb{R}^n)$, namely
\[
\| \chi u \|_{\DOM} = \| (\widetilde{P}(h)+i)(\chi u) \|_{L^2(\mathbb{R}^n)} \leq C \| u \|_{H^2(\mathbb{R}^n)}. 
\]
\end{proof}

In what follows, we will always be concerned with $\lambda$ ranging in a precompact neighborhood of $\mathbb{R}_+$. So fix $0< a_0<b_0, \, \epsilon_0>0$ and define
\[
\Omega(h) = (a_0,b_0) + i((-\gamma+\epsilon_0)h,1).
\]
For each $\epsilon >0$ we also define a shrunken neighborhood,
\[
\Omega_\epsilon(h) = (a_0+\epsilon,b_0-\epsilon) + i((-\gamma+\epsilon_0+\epsilon)h,1).
\]

\begin{prop} The function $R(\lambda,h)$ has a meromorphic continuation in $\Omega(h)$ as a family of bounded operators $\mathcal{H}_\gamma \rightarrow \mathcal{H}_{-\gamma}$.
\end{prop}  \label{prop:meromorphicextension}
\begin{proof}
Choose cutoff functions $\chi, \chi_i \in C_c^\infty(\mathbb{R}^n), \, i=0,1,2$, so that $\chi_0 \equiv 1$ near $\BINT$ with $\chi_i \equiv 1$ on $\supp \chi_{i-1}$ and $\chi \equiv 1$ on $\supp \chi_2$. We can always choose these so that $\chi \varphi = 0$ and $\chi_i \varphi \equiv 0$.
Approximate $R(\lambda,h)$ by a parametrix of the form $Q_0(\lambda,\lambda_0,h)+Q_1(\lambda_0,h)$
where
\begin{align*}
&Q_0(\lambda,\lambda_0,h) = (1-\chi_0)(R_0(\lambda,h)-\widetilde{R}(\lambda_0,h)A(h)R_0(\lambda,h))(1-\chi_1), \\
&Q_1(\lambda_0,h) = \chi_2 R(\lambda_0,h) \chi_1, 
\end{align*}
see also \cite{SaBarreto:1995}. Here $\lambda_0 = \lambda_0(h)$ denotes a point in $\Omega(h)$ with $\Im \lambda_0 \geq T_0 h$. We now compute
\[
(P(h)-\lambda^2)Q_0(\lambda,\lambda_0,h) = (1-\chi_1) + K_0(\lambda,\lambda_0,h) + K_1(\lambda,\lambda_0,h)
\]
where
\begin{align*}
K_0(\lambda,\lambda_0,h) &= -[\widetilde{P}(h),\chi_0](R_0(\lambda,h)-\widetilde{R}(\lambda_0,h)A(h)R_0(\lambda,h))(1-\chi_1)\\
K_1(\lambda,\lambda_0,h) &= (1-\chi_0)(\lambda^2-\lambda_0^2)(\widetilde{R}(\lambda_0,h)A(h)R_0(\lambda,h)(1-\chi_1),
\end{align*}
and
\[
(P(h)-\lambda^2)Q_1(\lambda_0,h) = \chi_1 + K_2(\lambda_0,h) +  K_3(\lambda,\lambda_0,h)
\]
where
\begin{align*}
K_2(\lambda_0,h) &= -[\widetilde{P}(h),\chi_2]R(\lambda_0,h)\chi_1, \\
K_3(\lambda,\lambda_0,h) &= \chi_2(\lambda_0^2-\lambda^2) R(\lambda_0,h) \chi_1.
\end{align*}
If we let $K =  K_0+ K_1 +  K_2 +  K_3$ then
\[
(P(h)-\lambda^2)(Q_0(\lambda,\lambda_0,h) + Q_1(\lambda_0,h)) = I + K(\lambda,\lambda_0,h).
\]
Note that if $\psi\in C^\infty_c(\mathbb{R}^n)$ then $[\widetilde{P}(h),\psi]$ is a first order operator with compactly supported coefficients and $\|[\widetilde{P}(h),\psi]\|_{H^1(\mathbb{R}^n)\rightarrow L^2(\mathbb{R}^n)} = O(h)$.

It is easy to see that $Q_0 + Q_1 : \mathcal{H}_\gamma\rightarrow \mathcal{H}_{-\gamma}$. For $Q_0$ this follows from the mapping properties of $R_0(\lambda,h), A(h)$, and $\widetilde{R}(\lambda_0,h)$. For $Q_1$ this fact is trivial since $Q_1$ contains compactly supported cutoffs. We also remark that by the resolvent identity,
\[
K_0(\lambda_0,\lambda_0,h) = -[\widetilde{P}(h),\chi_0] \widetilde{R}(\lambda_0,h)(1-\chi_1).
\]

To apply the Fredholm theory, we begin by showing that $K: \mathcal{H}_\gamma \rightarrow \mathcal{H}_\gamma$ is compact. First note that 
\[
K_0(\lambda,\lambda_0,h) = -[\widetilde{P}(h),\chi_0] (R_0(\lambda,h)-\widetilde{R}(\lambda_0,h)A(h)R_0(\lambda,h)) (1-\chi_1) : \mathcal{H}_\gamma \rightarrow \mathcal{H}_\gamma
\]
is compact: we see that $R_0(\lambda,h):L^2_\gamma(\mathbb{R}^n)\rightarrow H^2_{-\gamma}(\mathbb{R}^n)$ and $\widetilde{R}(\lambda_0,h)A(h)R_0(\lambda,h):L^2_\gamma(\mathbb{R}^n)\rightarrow H^2_{\gamma+\delta}(\mathbb{R}^n)$. On the other hand $[\widetilde{P}(h),\chi_0]$ is compactly supported and hence maps $H^2_\alpha(\mathbb{R}^n) \rightarrow L^2(\Bext)$ compactly for each $\alpha \in\mathbb{R}$. 

Similarly we can write
\[
K_2(\lambda_0,h) = [\widetilde{P}(h),\chi_2] (1-\chi_0) R(\lambda_0,h) \chi_1
\]
which is compact since $(1-\chi_0) R(\lambda_0,h)\chi_1 : \mathcal{H}_{\gamma} \rightarrow H^2(\Bext)$ and $[\widetilde{P}(h),\chi_2]$ is compactly supported. To see that $K_1$ is compact, again use that $\widetilde{R}(\lambda_0,h)A(h)R_0(\lambda,h):L^2_\gamma(\mathbb{R}^n)\rightarrow H^2_{\gamma+\delta}(\mathbb{R}^n)$ and now appeal to the fact that the inclusion
\[
H^2_{\gamma+\delta}(\mathbb{R}^n) \hookrightarrow L^2_{\gamma}(\mathbb{R}^n)
\]
is compact. Finally, the compactness of $K_3(\lambda,\lambda_0,h)$ follows from \eqref{eq:blackboxcompactness}.

Next, we need to verify the invertibility of $I+K(\lambda,\lambda_0,h)$ for at least one value of $\lambda\in\Omega(h)$. Recall that multiplication by $(1-\chi_0): H^2(\mathbb{R}^n)\rightarrow \DOM$ is uniformly bounded in $h$. It follows that for $\lambda_0\in \Omega(h)$ in the upper half-plane, we have
\[
\| (1-\chi_0) R(\lambda_0,h)u \|_{H^2(\mathbb{R}^n)} \leq C_1\|(P(h)+i)R(\lambda_0,h)u\|_{\mathcal{H}} \leq C_2 |\Im \lambda_0|^{-1}\|u\|_{\mathcal{H}},\, u\in \mathcal{H}
\]
and hence 
\[
\|(1-\chi)R(\lambda_0,h)\|_{\mathcal{H}\rightarrow H^2(\mathbb{R}^n)} = O(|\Im \lambda_0|^{-1}),\, \lambda_0 \in \Omega(h),\, \Im \lambda_0 > 0.
\] 
Here we used 
\[
(P(h)+i)R(\lambda_0,h) = I + (\lambda_0^2+i)R(\lambda_0,h)
\]
and $R(\lambda_0,h) = O_{\mathcal{H}\rightarrow\mathcal{H}}(|\Im \lambda_0|^{-1})$. Combining this with Equation \eqref{eq:tilderesolvent}, we see there exists $T_1>T_0$ such that if $\lambda_0\in \Omega(h)$ satisfies $\Im \lambda_0 \geq T_1h$, then
\[
\| K(\lambda_0,\lambda_0,h)\|_{\mathcal{H}_\gamma\rightarrow \mathcal{H}_\gamma} = O(h|\Im \lambda_0|^{-1}) \leq 1/2,
\]
and hence $I+K(\lambda_0,\lambda_0,h)$ will be invertible.

\end{proof}

\begin{rem} The poles and their multiplicities of the extension obtained above do not depend on the particular choice of $\varphi$. Indeed, if $\varphi_1$ and $\varphi_2$ both vanish near $\Bext$ and equal $|x|$ for $|x|$ large, then
\[
e^{-\gamma \varphi_1} R(\lambda,h) e^{-\gamma \varphi_1} = e^{-\gamma (\varphi_1-\varphi_2)} e^{-\gamma \varphi_2}  R(\lambda,h) e^{-\gamma \varphi_2} e^{-\gamma (\varphi_1-\varphi_2)}
\]
and vice versa. Hence the poles and multiplicities of one such extension agree with those of any other.
  
\end{rem}

\begin{rem} As pointed out by the anonymous referee, an interesting question is whether $R(\lambda,h)$ can be continued to a larger region in the lower half plane when the perturbations are smooth functions of $\exp((-2\gamma -\delta)|x|)$ for $|x|$ large (and also whether the corresponding resolvent estimates hold). Such hypotheses are satisfied for stationary wave operators arising from black hole metrics with nondegenerate event horizons, see \cite{Dyatlov:2011}, \cite{Gannot:2014} for two examples.  
\end{rem}

At this point we need to introduce a new assumption on a reference operator $P^\sharp(h)$, defined as follows: choose $R_1>R_0$ and $R_2 > 2R_1$ and let $\mathbb{T}$ denote the torus $\mathbb{T}=(\mathbb{R}/R_2\mathbb{Z})^n$. Let 
\[
\mathcal{H}^\sharp = \mathcal{H}_{R_0} \oplus L^2(\mathbb{T}\backslash \BINT)
\]
where $B(0,R_1)$ is considered a subset of $\mathbb{T}$. Define the dense subspace
\[
\DOM^\sharp = \{u\in \mathcal{H}^\sharp: \, \psi u\in \DOM \, , (1-\psi)u\ \in H^2(\mathbb{T})  \},
\]
where $\psi\in C_c^\infty (B(0,R_1))$ satisfies $\psi\equiv 1$ near $\BINT$. Now set
\[
P^\sharp(h)u = P(h)\psi u + (-\sum_{i,j} (h\partial_{i}) a_{ij} (h\partial_{j}) + V)(1-\psi)u, \quad u\in\DOM^\sharp.
\]
Then $P^\sharp(h)$ is self-adjoint on $\DOM^\sharp$ with discrete spectrum. We require that
\begin{equation} \label{eq:referencehypothesis}
\# \{z\in \sigma(P^\sharp(h)): z \in [-L,L] \}  \leq C(L/h^2)^{n^\sharp/2}
\end{equation}
for some $n^\sharp \geq n$ and each $L\geq 1$. Here the eigenvalues are counted with multiplicity. If $z_1,z_2,z_3,\ldots$ are the eigenvalues of $P^\sharp(h)$ ordered so $|z_1| \leq |z_2| \leq |z_3| \leq \ldots$, then the singular values of $(P^\sharp(h)-\lambda_0^2)^{-1}$ are $\mu_j((P^\sharp(h)-\lambda_0^2)^{-1}) = |z_j-\lambda_0^2|^{-1}$. If $\Im \lambda_0 = T_1h$, then Condition \eqref{eq:referencehypothesis} implies that there exists a constant $C>0$ so that 
\[
\mu_j((P^\sharp(h)-\lambda_0^2)^{-1}) \leq Ch^{-2}j^{-2/{n^{\sharp}}}, \quad j > Ch^{-n^\sharp}.  
\]

\section{Resolvent estimates}

To estimate $R(\lambda,h)$, we make use of the following general fact \cite[Chap. V, Theorem 5.1]{Gohberg:1967}: Suppose $A$ is a compact operator lying in some $p$-class. If $(I+A)$ is invertible then
\[
\| (I + A)^{-1} \| \leq \frac{\det(I+|A|^{p})}{|\det(I+A^p)|}.
\]
We wish to apply this inequality to $(I+K)$, but first we need to verify that a suitable power of $K$ is of trace class. Under our hypotheses we cannot estimate the singular values of $K_2$; nevertheless, the proof of Proposition \ref{prop:meromorphicextension} shows that $I+K_2(\lambda_0,h)$ is invertible on $\mathcal{H}_\gamma$ for $\Im \lambda_0 > T_1h$, so we use the decomposition
\[
(I+K(\lambda,\lambda_0,h))=(I+K_2(\lambda_0,h))(I+\widetilde{K}(\lambda,\lambda_0,h)),
\]
where $\widetilde{K} = (I+K_2)^{-1}(K_0+K_1+K_3)$. Note that $I+K$ and $I+\widetilde{K}$ have the same poles. 

\subsection{Singular values}
From now on we will always choose $\lambda_0 \in \Omega(h)$ with fixed imaginary part $\Im \lambda_0 = T_1 h$. Throughout, it will be clear that whenever an estimate depends on $\lambda_0\in \Omega(h)$, it really only depends on $\Im \lambda_0$.

\begin{prop}
The operator $\widetilde{K}(\lambda,\lambda_0,h)^{n^\sharp+1}: \mathcal{H}_\gamma \rightarrow \mathcal{H}_\gamma$ is of trace class for $\lambda\in \Omega(h)$.
\end{prop}
\begin{proof}
We estimate the singular values of each summand in $\widetilde{K}$. Since the weighted resolvent only continues to a narrow strip in the lower half-plane, in such a region it is particularly simple to estimate $\mu_j(K_0)$: Choose an open ball $B\subseteq \mathbb{R}^n$ containing $\supp \nabla \chi_0$ and let $-\Delta_B$ denote the Dirichlet Laplacian on $B$. Again using that the inclusion $1_{\Bext}:\DOM_\gamma\rightarrow H^{2}_\gamma$ is uniformly bounded in $h$, we consider $K_0$ as a map $\mathcal{H}_\gamma\rightarrow H^1(B)$. By Weyl asymptotics,
\[
\mu_j((-h^2\Delta_B)^{-1}) \leq C h^{-2} j^{-2/n}, \quad j =1,2,3,\ldots
\]  
Thus we estimate
\begin{align*}
\mu_j(K_0(\lambda,\lambda_0,h)) &\leq C \mu_j\left(\left(-h^2\Delta_B\right)^{-1/2}\right)\|(-h^2\Delta_B)^{1/2}K_0(\lambda,\lambda_0,h)\|_{\mathcal{H}_\gamma\rightarrow L^2(B)} \\
&\leq C h^{-3} j^{-1/n},\; \lambda\in \Omega(h).  
\end{align*}

By the same reasoning we estimate $\mu_j(K_1)$, writing
\[
\mu_j(K_1(\lambda,\lambda_0,h)) \leq C \mu_j(e^{\gamma\varphi} \widetilde{R}(\lambda_0,h) e^{-(\gamma+\delta)\varphi}) \|e^{(\gamma+\delta)\varphi} A(h) e^{\gamma\varphi} e^{-\gamma\varphi} R_0(\lambda) \|_{L^2_\gamma(\mathbb{R}^n)\rightarrow L^2(\mathbb{R}^n)}.
\]
In order to bound $\mu_j(e^{\gamma\varphi} \widetilde{R}(\lambda_0,h) e^{-(\gamma+\delta)\varphi})$, let $P_0(h) = -h^2\Delta + x^2$ denote the harmonic oscillator. The inequality $\mu_j(P_0(h)^{-1}) \leq C h^{-1}j^{-1/n}$ follows, in this case by explicit knowledge of the spectrum. Since $P_0(h)e^{-\delta\varphi}$ is bounded $H^2(\mathbb{R}^n)\rightarrow L^2(\mathbb{R}^n)$, we have
\begin{align*}
\mu_j(e^{\gamma\varphi} \widetilde{R}(\lambda_0,h) e^{-(\gamma+\delta)\varphi}) &\leq \mu_j(P_0(h)^{-1})\|P_0(h) e^{-\delta\varphi} e^{(\gamma+\delta)\varphi} \widetilde{R}(\lambda_0,h) e^{-(\gamma+\delta)\varphi}\|_{L^2(\mathbb{R}^n)\rightarrow L^2(\mathbb{R}^n)} \\
&\leq C h^{-2} j^{-1/n}.
\end{align*}
Combined with the previous estimate we obtain 
\[
\mu_j(K_1) \leq C h^{-3}j^{-1/n},\;\lambda\in \Omega(h).
\]
Next we estimate the singular values of $K_3$ using \eqref{eq:referencehypothesis}. First recall that $(P(h)-\lambda^2)\chi = (P^\sharp(h)-\lambda^2)\chi$, which implies that
\[
(P(h)-\lambda_0^2)^{-1}\chi_1 = \chi (P^\sharp(h)-\lambda_0^2)^{-1}\chi_1-(P(h)-\lambda_0^2)^{-1}[P^\sharp(h),\chi](P^\sharp(h)-\lambda_0^2)^{-1}\chi_1.
\]
Multiply this equation on the left by $\chi_2$ and apply Fan's inequality, $\mu_{2k-1}(A+B) \leq \mu_k(A) + \mu_k(B)$. Using that $(P(h)-\lambda_0^2)^{-1}[P^\sharp(h),\chi]$ has norm $O(1)$, we obtain
\[
\mu_j(K_3(\lambda,\lambda_0,h)) \leq C h^{-2} j^{-2/{n^\sharp}}, \quad j > Fh^{-n^\sharp} 
\]
for some constant $F>0$. For $j \leq Fh^{-n^\sharp}$ we simply bound $\mu_j(K_3) \leq Ch^{-1}$ using the trivial norm estimate. 

It is now clear that $\mu_j(K_i)^{n^\sharp}$ is summable for $j=0,1,3$.
\end{proof}

Applying the resolvent estimate as above, we obtain 
\begin{align} \label{eq:resolventestimate}
\| R(\lambda,h) \|_{\mathcal{H}_\gamma \rightarrow \mathcal{H}_{-\gamma}} &\leq \| Q_0 + Q_1 \|_{\mathcal{H}_\gamma \rightarrow \mathcal{H}_{-\gamma}} \| (I+K_2)^{-1} \|_{\mathcal{H}_\gamma \rightarrow \mathcal{H}_{\gamma}} \| (I+\widetilde{K})^{-1}\|_{\mathcal{H}_\gamma \rightarrow \mathcal{H}_{\gamma}} \notag\\
&\leq C \| Q_0 + Q_1 \|_{\mathcal{H}_\gamma \rightarrow \mathcal{H}_{-\gamma}} \frac{\det(I+(\widetilde{K}^*\widetilde{K})^\frac{n^\sharp+1}{2})}{|\det(I+\widetilde{K}^{n^\sharp+1})|}.
\end{align}

Since 
\[
\| Q_0 + Q_1 \|_{\mathcal{H}_\gamma \rightarrow \mathcal{H}_{-\gamma}} = O(h^{-2}),\; \lambda\in \Omega(h),
\]
it remains only to estimate the determinants. Define
\[
f(\lambda,h) = \det (I+\widetilde{K}^{n^\sharp+1}(\lambda,\lambda_0,h))
\]
in $\Omega(h)$. By Weyl convexity inequalities, it follows that $|f(\lambda,h)| \leq M(h), \, \lambda\in\Omega(h)$, where
\[
M(h) = \sup_{\lambda\in \Omega(h)} \det(I+(\widetilde{K}^*\widetilde{K})^{\frac{n^\sharp+1}{2}}).
\]
We therefore need to bound $M(h)$ from above and $|f(\lambda,h)|$ from below. 

\subsection{Estimating the determinant from above} 
Here we obtain an upper bound for $M(h)$ of the form $M(h) \leq e^{Ch^{-p}}$. For the application in mind, the value of $p$ is unimportant and we do not attempt to optimize the exponent. In fact $h^{-p}$ also represents a polynomial upper bound for the number of resonances in a disk of radius $h$, but again obtaining an optimal value is unimportant in this context.
  
\begin{prop} \label{prop:upperbound} There exists $C>0$ depending only on $\Im\lambda_0$, and $p>0$ such that
\[  
M(h) \leq e^{Ch^{-p}}.
\]
\end{prop}
\begin{proof}
We estimate $M(h)$ using Fan's inequalities:
\begin{align*}
\prod_{j\geq 1}(1+\mu_j(\widetilde{K}^{n^\sharp+1})) &= \prod_{j\geq 1}(1+\mu_j(\widetilde{K})^{n^\sharp+1}) \leq \prod_{j\geq 1}(1+\mu_{3j-2}(\widetilde{K})^{n^\sharp+1})^3 \\ 
&\leq \prod_{j\geq 1}(1+C_0 (\mu_j(K_0)^{n^\sharp+1}+\mu_j(K_1)^{n^\sharp+1}+\mu_j(K_3)^{n^\sharp+1}))^3\\
&\leq \prod_{i=0,1,3} \prod_{j\geq 1} (1+C_0\mu_j(K_i)^{n^\sharp+1})^3.
\end{align*}
For $i=0,1$, the singular values occuring in this product are bounded above by $\mu_j(K_i) \leq Ch^{-3}j^{-1/n^\sharp}$ and so we bound the product by the trace,
\begin{align*}
\prod_{j\geq 1} (1+C_0\mu_j(K_i)^{n^\sharp+1}) &\leq \exp(C_1 h^{-3n^\sharp-3} \sum_{j\geq1} j^{-1+1/n^\sharp}) \\ 
&\leq e^{Ch^{-3n^\sharp-3}}.
\end{align*}
On the other hand for $K_3$, we have
\begin{align*}
\prod_{j\geq 1} (1+C_0\mu_j(K_3)^{n^\sharp+1}) &\leq \prod_{1\leq j \leq Fh^{-n^\sharp}} (1+C_0\mu_j(K_3)^{n^\sharp+1}) \prod_{j > F h^{-n^\sharp}} (1+C_0\mu_j(K_3)^{n^\sharp+1}) \\
&\leq \left(e^{Ch^{-n^\sharp}\log(1/h)}\right)\left(e^{Ch^{-n^\sharp}}\right).
\end{align*}
Thus we get
\[
M(h) \leq e^{Ch^{-p}}
\]
for some $p>0$, where the constant $C$ has dependence only on $\Im\lambda_0$. 
\end{proof}

\subsection{Estimating the determinant from below}
Next we need to estimate $|f(\lambda,h)|$ from below. Note that $\lambda_0$ is not a zero of $f(\lambda,h)$ and that we have
\[
(I+\widetilde{K}(\lambda_0,\lambda_0,h)^{n^\sharp+1})^{-1} = I - \widetilde{K}(\lambda_0,\lambda_0,h)^{n^\sharp+1}(I+\widetilde{K}(\lambda_0,\lambda_0,h)^{n^\sharp+1})^{-1}.
\]
By taking determinants and arguing as in the previous section, we obtain a lower bound at $\lambda_0$,
\[
|f(\lambda_0,h)| \geq e^{-Ch^{-p}},
\]
where the constant again depends only on $\Im \lambda_0$. Since we can bound $|f(\lambda,h)|$ from above by $M(h)$ and from below at a chosen point, we are in a position to employ Cartan's principle \cite[Theorem 11]{Levin:1972} to obtain a lower bound away from resonances.

\begin{prop} \label{prop:lowerbound}
For each $\epsilon > 0$ there exists $C=C(\epsilon)$ such that
\[
|f(\lambda,h)| \geq e^{-Ah^{-p}\log(1/S(h))}, \quad\lambda \in \Omega_\epsilon(h) \backslash \bigcup_j D(r_j(h),S(h)),
\]
where $S(h) \ll 1$ and $\left\{r_j(h)\right\}$ denote the resonances of $P(h)$ in $\Omega_\epsilon(h)$.
\end{prop}

\begin{proof}
Rather than applying \cite[Theorem 11]{Levin:1972} directly, we prefer to control the set where the lower bound holds at the expense of the quality of the lower bound, just as in \cite{Petkov:2001}. For the reader's convenience we reproduce the proof, making the necessary adjustments.

Choose $\lambda_0$ with fixed real part. Define radii and disks
\[
\rho_s(h) = T_1 + \gamma - \epsilon_0 - s \epsilon, \quad  D_s(h) = D(\lambda_0,\rho_s(h)), \; s=1,2,3.
\]
We see that $f(\lambda,h)$ is analytic in the disk $D_1(h)$. Let $r_j(h),\, j=1,\ldots,N(h)$ denote the zeros of $f(\lambda,h)$ in $D_2(h)$ including multiplicity and define the Blaschke product
\[
\phi(\lambda,h) =  \frac{(-\rho_2(h))^{N(h)}}{(r_1(h)-\lambda_0)\cdots (r_{N(h)}(h)-\lambda_0)} \prod_j \frac{\rho_2(h)(\lambda-r_j(h))}{\rho_2(h)^2 - (\overline{r_j(h)-\lambda_0})(\lambda-\lambda_0)}.
\]
Then $\phi$ has the same zeros as $f(\lambda,h)$, no poles in $D_2(h)$, and satisfies $\phi(\lambda_0,h)=1$. Moreover, on the boundary of $D_2(h)$ we see that
\begin{equation} \label{eq:philowerbound}
|\phi(\lambda,h)| = \frac{\rho_2^{N(h)}(h)}{|(\lambda_0-r_1(h)) \cdots (\lambda_0-r_{N(h)})|} \geq 1. 
\end{equation}
Since the function defined by
\[
\psi(\lambda,h) = f(\lambda,h)/\phi(\lambda,h)
\]
has no zeros in $D_2(h)$, we may apply Caratheodory's estimate \cite[Theorem 8]{Levin:1972} and \eqref{eq:philowerbound} to conclude that in $D_3(h)$ we have the lower bound 
\begin{align*}
\log |\psi(\lambda,h)|&\geq -\frac{2\rho_3(h)}{\epsilon} \log \left(\sup_{\lambda\in D_2(h)}|\psi(\lambda,h)| \right) + \frac{\rho_2(h)+\rho_3(h)}{\epsilon}\log(|\psi(\lambda_0,h)|) \\
&\geq -\frac{2\rho_3(h)}{\epsilon} \log \left(\sup_{\lambda\in D_2(h)} |f(\lambda,h)| \right)+\frac{\rho_2(h)+\rho_3(h)}{\epsilon}\log(|f(\lambda_0,h)|).
\end{align*}
It therefore suffices to bound $|\phi(\lambda,h)|$ from below in $D_3(h)$. 

Outside the set $\bigcup_j D(r_j(h),S(h))$, the polynomial appearing in the numerator of $\phi(\lambda,h)$ is  bounded below by $(S(h))^{N(h)}$. On the other hand, the polynomial in the denominator of $\phi(\lambda,h)$ is bounded above in $D_3(h)$ by $\rho_2(h)^{N(h)}(\rho_2(h)+\rho_3(h))^{N(h)}$. Therefore
\[
|\phi(\lambda,h)| \geq \left(\frac{S(h)}{\rho_2(h)(\rho_2(h)+\rho_3(h))}\right)^{N(h)}, \, \lambda\in D_3(h)\backslash \bigcup_j D(r_j(h),S(h)).
\]
Moreover, we can apply Jensen's formula to estimate the number of zeros $N(h)$ in $D_2(h)$ by
\begin{align*}
N(h) &\leq \frac{1}{\log\left(\frac{\rho_1(h)}{\rho_2(h)}\right)}\left(\log \left(\sup_{\lambda\in D_1} |f(\lambda,h)|\right)-\log |f(\lambda_0,h)|\right)\\
&\leq  \frac{1}{\log\left(\frac{\rho_1(h)}{\rho_2(h)}\right)} \left( \log M(h) - \log |f(\lambda_0,h)| \right) = O(h^{-p}).
\end{align*}
Combining all the contributions, we obtain 
\[
|f(\lambda,h)| \geq e^{-C h^{-p}\log(1/S(h))}, \, \lambda\in D_3(h)\backslash \bigcup_j D(r_j(h),S(h)).
\]
Since all the constants appearing are uniform in $\Re \lambda_0$, we can vary the real part in $\Omega_\epsilon(h)$ and obtain the necessary lower bound. Of course $\epsilon$ is arbitrary and the result follows.
\end{proof} 

We can now establish our main theorem on resolvent estimates.

\begin{theo} \label{theo:aprioribound}
For each $\epsilon > 0$ there exists $A=A(\epsilon)$ such that
\[
\| R(\lambda,h) \|_{\mathcal{H}_\gamma \rightarrow \mathcal{H}_{-\gamma}} < e^{Ah^{-p}\log(1/S(h))}, \quad\lambda \in \Omega_\epsilon(h) \backslash \bigcup_j D(r_j(h),S(h)),
\]
where $S(h) \ll 1$ and $\left\{r_j(h)\right\}$ denote the resonances of $P(h)$ in $\Omega_\epsilon(h)$.
\end{theo}
\begin{proof}
The proof follows immediately by applying Propositions \ref{prop:upperbound}, \ref{prop:lowerbound} to Equation \eqref{eq:resolventestimate}.
\end{proof}

\section{From quasimodes to resonances}

The passage from quasimodes to resonances is essentially an argument by contradiction. In the absence of resonances, the exponential bound appearing in Theorem \ref{theo:aprioribound} would hold throughout $\Omega_\epsilon(h)$; combined with the self-adjoint bound in the upper half-plane, an application of the ``semiclassical maximum principle'' implies a resolvent estimate on the real axis that contradicts the existence of a real quasimode. 
 First results in this direction are due to Stefanov--Vodev \cite{Stefanov:1996} who used the Phragm\'en--Lindel\"of principle to show that 
 having high energy real quasimodes implies existence of 
 resonances converging to the real axis. Bounds on the resolvent play a central role in that argument which go back to the work 
 of Carleman \cite{Carleman:1936} on the completeness of sets of eigenfunctions. Tang--Zworski \cite{Tang:1998} replaced the Phragm\'en--Lindel\"of principle with a local version of the maximum principle which showed that there exists a resonance close to each quasimode. Stefanov further refined these method  by dealing with multiplicities \cite{Stefanov:1999}, and modifying the maximum principle to allow the localization of resonances exponentially close to the real axis \cite{Stefanov:2005}.

\subsection{Quasimodes}
Suppose that $u(h)\in\DOM$ satisfies $\| u(h) \| = 1$ and
\[
\supp u(h) \subset K \text{ for a compact set $K$ independent of $h$}.
\]
Suppose further that there exists $\lambda(h)^2 \in \left(a_0,b_0\right)$ such that 
\[
\|(P(h)-\lambda(h)^2)u(h)\| \leq R(h)
\]
for a function $R(h)\geq 0$. We refer to such functions as quasimodes with accuracy $R(h)$. For the resolvent, choose a weight $\varphi$ so that $\varphi \equiv 0$ on $K$. Also choose $\chi_1$ with $\varphi\equiv 0$ on $\supp \chi_1$ and $\chi_1 \equiv 1$ on $K$. Notice that for $\lambda$ in the upper half-plane,
\[
e^{-\gamma\varphi}R(\lambda,h)e^{-\gamma\varphi}(P(h)-\lambda^2)u =e^{-\gamma\varphi}R(\lambda,h)e^{-\gamma\varphi}(P(h)-\lambda^2) \chi_1 u = u.  
\]
and hence this equation holds away from poles by analytic continuation. We also recall the following standard fact: consider the Laurent expansion of $e^{-\gamma\varphi}R(\lambda,h)e^{-\gamma\varphi}$ near a resonance $r(h)$,
\[
e^{-\gamma\varphi}R(\lambda,h)e^{-\gamma\varphi} = \textrm{holomorphic}(\lambda) + \sum_{j=1}^{N} A_j (\lambda^2-r(h)^2)^{-j}.
\]
Then $\mathrm{range}(A_j)\subseteq \mathrm{range}(A_1)$ for $j=1,\ldots,N$. For a very general discussion of these types of results, see \cite{Agmon:1998}. So consider the resonances $r_i(h)$ for $i=1,\ldots,N(h)$ contained in the set $\Omega_\epsilon(h)$, each with the associated residue $A^{(i)}_1$. If $\Pi$ denotes the projection onto $\oplus_i \mathrm{range} (A^{(i)}_1)$, then $(I-\Pi)A^{(i)}_j = 0$ for each $i,j$. Hence 
\[
(I-\Pi) e^{-\gamma\varphi}R(\lambda,h)e^{-\gamma\varphi}
\]
is holomorphic in $\Omega_\epsilon(h)$. By the maximum principle, this operator satisfies the bound given by Theorem \ref{theo:aprioribound} in a set slightly smaller than $\Omega_\epsilon(h)$ (see the proof of \cite[Theorem 1]{Stefanov:1999} or \cite[Theorem 3]{Stefanov:2005} for a precise statement).

\subsection{Semiclassical maximum principle} We now review the semiclassical maximum principle as presented in \cite{Stefanov:2005}.

\begin{lem}
Let $a(h)<b(h)$ and suppose $S_\pm(h), \alpha(h), w(h)$ are functions satisfying
\[
0 < S_+(h) \leq S_-(h), \quad 1 \leq \alpha(h), \quad S_-(h)\alpha(h)\log \alpha(h) \leq w(h).
\]
Furthermore suppose $F(\lambda,h)$ is a holomorphic function defined in a neighborhood of
\[
[a(h)-w(h), b(h)+w(h)] + i[-\alpha(h)S_-(h),S_+(h)].
\]
If
\begin{align*}
|F(\lambda,h)| &\leq e^{\alpha(h)}, \quad\lambda\in [a(h)-w(h), b(h)+w(h)] + i[-\alpha(h)S_-(h),S_+(h)], \\
|F(\lambda,h)| &\leq M(h), \quad \lambda \in [a(h)-w(h), b(h)+w(h)] + iS_+(h),
\end{align*}
with $M(h) \geq 1$, then there exists $h_1 = h_1(S_-,S_+,\alpha) > 0$ such that
\[
|F(\lambda,h)| \leq e^3 M(h), \quad \lambda \in [a(h), b(h)] + i[S_-(h),S_+(h)]
\]
for $h\leq h_1$.
\end{lem}

For our application, we will apply this lemma with 
\begin{itemize}
\item $S_-(h) = S_+(h) = S(h)$,
\item $F(\lambda,h) = (I-\Pi)e^{-\gamma\varphi}R(\lambda,h)e^{-\gamma\varphi}$,
\item $\alpha(h) = Ch^{-p}\log(1/S(h))$,
\item $M(h) = 1/S(h)$.
\end{itemize} 
The choice of $S(h)$ and $w(h)$ is made as in \cite{Stefanov:2005} according to the accuracy $R(h)$ of the quasimodes. 

\subsection{Lower bounds on the number of resonances}
Here we state the main theorem on the existence of resonances rapidly converging to the real axis. We refer to \cite[Theorem 3]{Stefanov:2005} for the proof; the only modification is that instead of a compactly truncated resolvent $(I-\Pi)\chi R(\lambda,h) \chi$ we use $(I-\Pi)e^{-\gamma\varphi} R(\lambda,h)e^{-\gamma\varphi}$.

\begin{theo} \label{thm:existenceofres}
Let $P(h)$ satisfy the black box hypotheses. Let $0<a_0<a(h)<b(h)<b_0<\infty$. Assume there is an $h_0$ such that for $h<h_0$ there exists $m(h)\in\left\{1,2,\ldots\right\},\ \lambda_n(h)^2\in \left[a(h),b(h)\right]$, and $u_n(h)\in\DOM$ with $\|u_n(h)\|=1$ for $1\leq n\leq m(h)$ such that $\supp u_n(h) \subset K$ for a compact set $K$ independent of $h$. Suppose further that
\begin{enumerate} \itemsep8pt
	\item $\|(P(h)-\lambda_n(h)^2)u_n(h)\| \leq R(h)$,
	\item Whenever a collection $\left\{v_n(h)\right\}_{n=1}^{m(h)}\subset \mathcal{H}$ satisfies $\|u_n(h)-v_n(h)\| < h^N/M$, then $\left\{v_n(h)\right\}_{n=1}^{m(h)}$ are linearly independent,
\end{enumerate}
	where $R(h) \leq h^{p+N+1}/C\log(1/h)$ and $C\gg 1,\, N\geq 0,\, M>0$. Then there exists $C_0>0$ depending on $a_0,b_0$ and the operator $P(h)$ such that for $B>0$ there exists $h_1<h_0$ depending on $A,B,M,N$ so that the following holds: Whenever $h\in (0,h_1)$, the operator $P(h)$ has at least $m(h)$ resonances in the strip
\[
	\left[a(h)-c(h)\log\frac1h, b(h)+c(h)\log\frac1h\right] - i\left[0,c(h)\right]
\]
	where $c(h) = \max(C_0BMR(h)h^{-p-N-1}, e^{-B/h})$.
\end{theo}

\section*{Acknowledgements}
I would like to thank Maciej Zworski for suggesting the problem, along with many valuable conversations. Thanks to the anonymous referee for suggesting improvements in the exposition. I am also grateful to Jeffrey Galkowski for his interest in the problem and some helpful discussions.

\end{document}